\documentclass[12pt]{amsart}
\usepackage[colorlinks=true, pdfstartview=FitV, linkcolor=blue, citecolor=blue, urlcolor=blue]{hyperref}

\usepackage{amssymb,amsmath, amscd}
\usepackage{times, verbatim}
\usepackage{graphicx}
\usepackage[english]{babel}
 \usepackage[usenames, dvipsnames]{color}
\usepackage{amsmath,amssymb,amsfonts}
\usepackage{enumerate}
\usepackage{csquotes}
\usepackage{tabularx}
\usepackage{caption}
\usepackage{subcaption}
\usepackage{yhmath}
\usepackage{array}

\usepackage{tikz}
\usepackage{slashbox}
\usetikzlibrary{calc, arrows, decorations.markings} \tikzset{>=latex}
\usepackage{amsmath,amssymb,amsfonts}
\usepackage{amscd, amssymb, latexsym, amsmath, amscd}
\usepackage[all]{xy}
\usepackage{pb-diagram}
\usepackage{enumerate}
\usepackage{ulem}
\usepackage{anysize}
\marginsize{3cm}{3cm}{3cm}{3cm}
\input xy
\xyoption{all}
\usepackage{pb-diagram}
\usepackage[all]{xy}
\usepackage{array}
\usepackage{upgreek}

\usepackage{xcolor}
\input xy
\xyoption{all}
\usepackage{ytableau}
\usepackage{mathtools}
\usepackage[overload]{empheq}
\usepackage{dynkin-diagrams}

\usepackage{tabularx} 
\usepackage[vlines]{tabularht}
\usepackage{adjustbox}

\usepackage{graphicx} 
\usepackage{lipsum} 
\usepackage{booktabs} 

\theoremstyle{plain}
\newtheorem{theorem}{Theorem}[section]
\newtheorem{lemma}[theorem]{Lemma}

\newtheorem{corollary}[theorem]{Corollary}

\newtheorem{proposition}[theorem]{Proposition}

\theoremstyle{definition} \theoremstyle{definition}
\newtheorem{remark}[theorem]{Remark}
\newtheorem{example}[theorem]{Example}
\newtheorem{defn}[theorem]{Definition}

\theoremstyle{remark}

\newcommand{\G}{\textsc{\G}}

\newcommand{\Z}{\mathbb{Z}}

\newcommand{\R}{\mathbb{R}}

\newcommand{\C}{\mathbb{C}}

\def\G{{\rm G}}

\def\Sp{{\rm Sp}}

\def\GL{{\rm GL}}

\def\SO{{\rm SO}}

\def\gln{\GL(n)}
\def\spn{\Sp(2n)}
\def\sop{\SO(p)}

\def\glm{\GL(m)}
\def\spm{\Sp(2m)}
\def\soq{\SO(q)}

\newcommand{\wg}{W_{\mathfrak{g}}}

\newcommand{\wk}{W_{\mathfrak{k}}}

\newcommand{\lambar}{\underline{\lambda}}
\newcommand{\mubar}{\underline{\mu}}

\newcommand{\pilam}{\Pi_{\lambar}}

\newcommand{\psimu}{\Psi_{\mubar}}

\newcommand{\xibar}{\underline{\xi}}

\newcommand{\mubard}{\mubar^{'}}

\newcommand{\Sn}{\mathfrak{S}_n}

\newcommand{\glnm}{{\GL}^{n}_{m}}
\newcommand{\spnm}{{\Sp}^{2n}_{2m}}
\newcommand{\sopq}{{\SO}^{p}_{q}}
\newcommand{\soodq}{{\SO}^{2n+1}_{q}}
\newcommand{\soevq}{{\SO}^{2n}_{q}}
\newcommand{\soodod}{{\SO}^{2n+1}_{2m+1}}
\newcommand{\soodev}{{\SO}^{2n+1}_{2m}}
\newcommand{\soevod}{{\SO}^{2n}_{2m+1}}
\newcommand{\soevev}{{\SO}^{2n}_{2m}}

\newcommand{\fracnom}{\genfrac{\{}{\}}{0pt}{}}

\begin{document}

\title[Branching Formulae for Classical Groups]
{New Branching Formulae for Classical Groups and Relations among them}

\begin{abstract}
We find the branching laws for the classical pairs $\GL(m, \C) \subset \GL(n, \C)$, $\Sp(2m, \C) \subset \Sp(2n, \C)$, $\SO(q, \C) \subset \SO(p, \C)$ for all $m\leq n$, and all $q\leq p$, generalizing the well-known results of classical branching laws which exist for $m=n-1$, and $q=p-1$. Our approach provides a common proof applicable to all these groups. We also compare the branching multiplicities among these pairs. 
\end{abstract}

\author{Dibyendu Biswas}

\address{Indian Institute of Technology Bombay, Powai, Mumbai, India-400076}

\email{dibubis@gmail.com}
\maketitle
    {\hfill \today}
    
\tableofcontents

\section{Introduction}
Branching rules are descriptions of how irreducible representations of a group $G$ decompose under restriction to a subgroup $H$. We are interested in the cases when $G$ and $H$ are  classical groups over complex numbers.
  We take $G$ to be general linear groups $\gln=\GL(n, \C)$, symplectic groups $\spn=\Sp(2n, \C)$, or orthogonal groups $\sop=\SO(p, \C)$.  We consider the following pair of groups $H \subset G$ in this paper,
 \[\begin{aligned}
 	& \GL(m) \subset \GL(n), && 
 	\Sp(2m) \subset \Sp(2n), &&&
 	\SO(q) \subset \SO(p).
 \end{aligned}\]
 We provide formulae (see Theorem~\ref{general gl} - \ref{general ortho}) expressing branching multiplicities  as determinants of certain combinatorial matrices. They are usually not multiplicity-free. We use the Weyl Character formula to prove all the formulae for branching multiplicities as mentioned here. The notation used in this paper as well as the proofs follow \cite[Chapter~8]{GW}.
 
 Our proof of the branching laws depends on a different version of the Weyl dimension formula which is our Theorem~\ref{new weyl dim}, and is a determinantal formula, to which our branching laws reduce to when branching from any group G to the trivial subgroup.

  The results for the pair of classical groups $(H \subset G)$ listed above for the special case  $m=n-1$ and $q=p-1$ have been known for a long time. In his book \cite[V.18]{HWqm}, H. Weyl provided the classical branching description for the pair $\GL(n-1) \subset \GL(n)$.  Concerning the pair $\Sp(2n-2) \subset \Sp(2n)$,  Zelobenko~\cite{ZDP} and Hegerfeldt~\cite{HGC} have established conditions under which the multiplicity is non-zero. The multiplicity formula in this case is due to Whippman~\cite{WML} (for $n=2,3$) and Miller~\cite{MW} (for the general case of $n$). For the pair $\mathrm{SO}(p-1) \subset \mathrm{SO}(p)$, the branching rules are in the book \cite[IX.9]{mur} by F.D. Murnaghan.

After finishing the paper, when we sent it to Prof. Okada, he informed us that Theorems~\ref{general gl}-\ref{general ortho} can be derived from an unpublished preprint \cite{okada89} of his from 1989. He proved his result using Lindström–Gessel–Viennot lemma which counts the number of tuples of non-intersecting lattice paths. He also mentioned the paper \cite{scrim} for similar results.

\section{The Main Theorems}\label{sec main thm}
In this section, we provide the statements of our main theorems. The theorems involve certain binomial coefficients, which we will define before going into the details.
\begin{defn}
Let $k\geq0$ be a integer and $x \in \R$. We define binomial coefficients $\binom{x}{k}$ as follows: 
\[\binom{x}{k}=\begin{cases}
		\frac{x(x-1)(x-2)\cdots (x-k+1)}{k!}, & \text{ if } k\geq 1 \\
	\hspace*{1.5cm}	1, & \text{ if } k=0.
\end{cases}
 \]   
\end{defn}
\begin{defn}\label{def fracnom}
	Let $k\geq0$ be a integer and $n \in \Z$. We define  $\fracnom{n}{k}$ as follows: 
	$$\fracnom{n}{k}= \begin{cases}
		\binom{n}{k}, & \text{ if } n \geq k, \\
		\hspace*{2mm}	0,  & \text{ if } n < k.  
	\end{cases}$$
\end{defn}

The dominant weights of the groups  $\gln$, $\spn$, $\SO(2n+1)$, $\SO(2n)$ are parameterized by sequences of integer $\lambar=\left(\lambda_1, \ldots, \lambda_n\right) \in \mathbb{Z}^n$ satisfying the following conditions.
$$
\begin{array}{ll}
	\lambda_1 \geq \lambda_2 \geq \cdots \geq \lambda_n, & \text{ for }  \gln , \\
	\lambda_1 \geq \lambda_2 \geq \cdots \geq \lambda_n \geq 0, & \text{ for } \Sp(2n) ,  \SO(2n+1), \\
	\lambda_1 \geq \lambda_2 \geq \cdots \geq \lambda_{n-1} \geq\left|\lambda_n\right|, & \text{ for } \SO(2n).
\end{array}
$$
For our branching problem, it is sufficient to consider $\lambda_n \geq 0$ for all groups which we tacitly assume everywhere. All the groups $\gln$, $\spn$, $\SO(2n+1)$, $\SO(2n)$ have rank $n$.

 Let $H \subset G$ be one of the pairs as in the Introduction. Let
		\[\lambar = (\lambda_1 \geq \lambda_2 \geq \cdots \geq \lambda_n \geq 0)
		\quad \text{ and } \quad
		\mubar = (\mu_1 \geq \mu_2 \geq \cdots \geq \mu_{m} \geq 0)\]
		be two sequences of integers. In case $(H \subset G)=\left(\SO(q) \subset \SO(p)\right)$, then $n = \lfloor \frac{p}{2} \rfloor$ and $m = \lfloor \frac{q}{2} \rfloor$. Let $\pilam$ and $\psimu$ denote the irreducible highest weight representations of $G$ and $H$ with the highest weights $\lambar$ and $\mubar$, respectively. Consider the following restriction, 
		\begin{equation}\label{restrict}
			{\pilam}|_H \hspace{.3cm} = \sum_{\mubar} m(\lambar, \mubar) \psimu.
		\end{equation}

 Set $\lambda_{n+i} = 0, \mu_{m+i} = 0 \text{ for } i \geq 1$.
 Let	\[	u_{ij}=\lambda_i - \mu_j + j -i \quad \text{ for } \hspace{.1cm} 1 \leq i,j \leq n.  \]
	
\begin{theorem}\label{general gl}
	Let  $\glm \subset \gln$, with $0 \leq m \leq n-1$. Then,
	\begin{enumerate}
		\item[(i)] The multiplicity
		$m(\lambar, \mubar)$ is nonzero if and only if 
		\begin{equation*}
			\lambda_i \geq \mu_i \geq \lambda_{i+n-m} \quad \text{ for } \quad 1 \leq i \leq m.
		\end{equation*}
		\item[(ii)]  
		The multiplicity $$m(\lambar, \mubar) =\det \left[M_{ij}\right],$$
		 where 
		\begin{equation*}
			M_{ij} = \begin{cases}
				\displaystyle\fracnom{u_{ij} + n-m-1}{n-m-1}, & \text{ if } \quad \hspace*{8.5mm} 1 \leq j \leq m, \vspace{2mm}\\
				\displaystyle\binom{u_{ij} + n-j}{n-j}, & \text{ if } \quad  m+1 \leq j \leq n.
			\end{cases} 
		\end{equation*}
	\end{enumerate}
\end{theorem}

\begin{theorem}\label{general sp}
	Let $\spm \subset \spn$, with $0 \leq m \leq n-1$.
	Then,
	\begin{enumerate}
		\item[(i)] The multiplicity 
		$m(\lambar, \mubar)$ is nonzero if and only if 
		\begin{equation*}
			\lambda_i \geq \mu_i \geq \lambda_{i+2n-2m} \quad \text{ for } \quad 1 \leq i \leq m.
		\end{equation*}
		
		\item[(ii)]  
		The multiplicity
		 $$m(\lambar, \mubar) =\det \left[M_{ij}\right],$$
		  where 
		\begin{equation*}
			M_{ij} = \begin{cases}
				\displaystyle\fracnom{u_{ij} + 2n-2m-1}{2n-2m-1}, & \text{ if } \quad \hspace*{.85cm} 1 \leq j \leq m, \vspace{2mm}\\
				\displaystyle\binom{u_{ij} + 2n-2j+1}{2n-2j+1}, & \text{ if } \quad m+1 \leq j \leq n.
			\end{cases} 
		\end{equation*}
	\end{enumerate}
\end{theorem}

\begin{theorem}\label{general ortho}
	Let $\soq \subset \sop$ with $0 \leq q \leq p-1$. Let $n = \lfloor \frac{p}{2} \rfloor$ and $m = \lfloor \frac{q}{2} \rfloor$. We define  $l$ as follows: 
	\[ l= \begin{cases}
		n-m, & \text{ if } \quad p=2n+1, \\
		n-m-1, & \text{ if } \quad p=2n.
	\end{cases} \]
	Then,
	\begin{enumerate}
		\item[(i)] The multiplicity 
		$m(\lambar, \mubar)$ is nonzero if and only if 
		\begin{equation*}
			\lambda_i \geq \mu_i \geq \lambda_{i+p-q} \quad \text{ for } \quad 1 \leq i \leq m,
		\end{equation*}
		
		\item[(ii)]  
		The multiplicity 
		$$m(\lambar, \mubar) =2^{l} \det \left[M_{ij}\right],$$ where 
		\begin{equation*}
			M_{ij} = \begin{cases}
				\displaystyle\fracnom{u_{ij} + p-q-1}{p-q-1}, & \text{ if } \quad \hspace*{.85cm} 1 \leq j \leq m, \vspace{2mm}\\
			\displaystyle\binom{u_{ij} + p-2j-\frac{1}{2}}{p-2j}, & \text{ if } \quad m+1 \leq j \leq n.
			\end{cases} 
		\end{equation*}
	\end{enumerate}
\end{theorem}

\begin{remark}
	 Putting $m = 0$ in Theorem~\ref{general gl} - \ref{general ortho} gives the corresponding Weyl dimension formula (Theorem~\ref{new weyl dim}) for general linear, symplectic and orthogonal groups, respectively. In fact, we first prove these dimension formulae (different looking than the usual Weyl dimension formula) which goes into the proof of Theorem~\ref{general gl} - \ref{general ortho}.
	 \end{remark}

\begin{remark}
	Put $m=n-1$ in Theorem~\ref{general gl}.  The inequality $\lambda_i \geq \mu_i \geq \lambda_{i+1}$ gives rise to $\lambda_1 \geq \mu_1 \geq \cdots \geq \lambda_{n-1} \geq \mu_{n-1} \geq \lambda_n$.  Applying this condition, we get $u_{ii}=1$ for $1\leq i \leq n$, while $u_{ij}<0$ for $1\leq j <i \leq n$. Therefore, the corresponding matrix becomes an upper triangular matrix where all the diagonal entries are equal to 1. Hence the multiplicity is 1 when we have the interlacing condition. So, we obtain the branching rule for $\GL(n-1) \subset \gln$. 
	A similar argument can be given for other pairs.
\end{remark}

 We use the following notation for pairs:
\[\begin{array}{rl}
	\begin{aligned}
		\glnm \, &:=\\
		\spnm \, &:=\\
		\sopq \, &:=
	\end{aligned} & \begin{aligned}
		\left( \right. \GL(m) &\subset \GL(n) \left. \right),\\
		\left( \right. \Sp(2m) &\subset \Sp(2n) \left. \right),\\
		\left( \right. \SO(q) &\subset \SO(p) \left. \right).
	\end{aligned}
\end{array}\]
\begin{example}We derive explicit multiplicity formulae as determinants from Theorem~\ref{general gl}-\ref{general ortho} for all pairs $H \subset G$, with the condition that $\text{rank }(G)=3$ and $\text{rank }(H)=1$. We use $|a_{ij}|$ to represent the determinant of the matrix  $ \begin{bmatrix}
		a_{ij}
	\end{bmatrix}$ in the following table.

  \begin{center}
	\newcolumntype{C}[1]{>{\centering\let\newline\\\arraybackslash\vspace{5pt}}m{#1}}
	\begingroup
	\setlength{\tabcolsep}{8pt} 
	\renewcommand{\arraystretch}{2} 
	\captionof{table}{Multiplicity Formula for $H \subset G$ with $\text{rk }(G)=3$, $\text{rk }(H)=1$.}
	\label{tab hihi}
	\begin{tabular}{|C{.2in}|C{2.3in}||c|C{1.9in}|}
		\hline
		Pair &  \setlength{\baselineskip}{10pt} $m(\lambar,\mubar)$ & Pair &  \setlength{\baselineskip}{10pt} $m(\lambar,\mubar)$  \\
		\hline \hline  
		$\Sp^6_2$ & \setlength{\baselineskip}{15pt} $\begin{vmatrix}
			\fracnom{\lambda_1 - \mu_1 +3}{3}&\binom{\lambda_1  +4}{3}& \binom{\lambda_1  +3}{1} \\ 
			\fracnom{\lambda_2 - \mu_1 +2}{3}&\binom{\lambda_2  +3}{3}& \binom{\lambda_2  +2}{1}\\
			\fracnom{\lambda_3 - \mu_1 +1}{3}&\binom{\lambda_3  +2}{3}& \binom{\lambda_3  +1}{1}
		\end{vmatrix}$ &$\GL^3_1$ & $\begin{vmatrix}
			\fracnom{\lambda_1 - \mu_1 +1}{1}&\binom{\lambda_1  +2}{1}& 1 \\ 
			\fracnom{\lambda_2 - \mu_1 }{1}&\binom{\lambda_2  +1}{1}& 1 \\
			\fracnom{\lambda_3 - \mu_1 -1}{1}&\binom{\lambda_3  }{1}& 1
		\end{vmatrix}$  \\  \hline 
		$\SO^7_3$ & $2^2 \cdot \begin{vmatrix}
			\fracnom{\lambda_1 - \mu_1 +3}{3}&\binom{\lambda_1  +3\frac{1}{2}}{3}& \binom{\lambda_1  +2\frac{1}{2}}{1} \\ 
			\fracnom{\lambda_2 - \mu_1 +2}{3}&\binom{\lambda_2  +2\frac{1}{2}}{3}& \binom{\lambda_2  +1\frac{1}{2}}{1}\\
			\fracnom{\lambda_3 - \mu_1 +1}{3}&\binom{\lambda_3  +1\frac{1}{2}}{3}& \binom{\lambda_3  +\frac{1}{2}}{1}
		\end{vmatrix}$ & $\SO^6_2$ & $2 \cdot\begin{vmatrix}
			\fracnom{\lambda_1 - \mu_1 +3}{3}&\binom{\lambda_1  +2\frac{1}{2}}{2}& 1 \\ 
			\fracnom{\lambda_2 - \mu_1 +2}{3}&\binom{\lambda_2  +1\frac{1}{2}}{2}& 1\\
			\fracnom{\lambda_3 - \mu_1 +1}{3}&\binom{\lambda_3  +\frac{1}{2}}{2}& 1
		\end{vmatrix}$ \\ \hline 
		$\SO^7_2$& 	$ 2^2 \cdot \begin{vmatrix}
			\fracnom{\lambda_1 - \mu_1 +4}{4}&\binom{\lambda_1  +3\frac{1}{2}}{3}& \binom{\lambda_1  +2\frac{1}{2}}{1} \\ 
			\fracnom{\lambda_2 - \mu_1 +3}{4}&\binom{\lambda_2  +2\frac{1}{2}}{3}& \binom{\lambda_2  +1\frac{1}{2}}{1}\\
			\fracnom{\lambda_3 - \mu_1 +2}{4}&\binom{\lambda_3  +1\frac{1}{2}}{3}& \binom{\lambda_3  +\frac{1}{2}}{1}
		\end{vmatrix} $  & $\SO^6_3$	& $ 2 \cdot \begin{vmatrix}
			\fracnom{\lambda_1 - \mu_1 +2}{2}&\binom{\lambda_1  +2\frac{1}{2}}{2}& 1 \\ 
			\fracnom{\lambda_2 - \mu_1 +1}{2}&\binom{\lambda_2  +1\frac{1}{2}}{2}& 1 \\
			\fracnom{\lambda_3 - \mu_1 }{2}&\binom{\lambda_3  +\frac{1}{2}}{2}& 1
		\end{vmatrix} $   \\
		\hline	\end{tabular}
	\endgroup
	\vspace{3mm}
\end{center}

	Using Table~\ref{tab hihi} we calculate,
	$$ \Pi_{(2,1,0)}|^{Sp(6)}_{Sp(2)} = {20} \times \Psi_{(0)} \ + \ 16 \times \Psi_{(1)} \ + \ {4} \times \Psi_{(2)},$$ 
	where the multiplicities are given by the following determinant calculation:
	$$20= \begin{vmatrix}
		10 & 20 & 5 \\
		1 & 4 & 3 \\
		0 & 0 & 1
	\end{vmatrix}, \quad 16= \begin{vmatrix}
	4 & 20 & 5 \\
	0 & 4 & 3 \\
	0 & 0 & 1
	\end{vmatrix}, \quad 4= \begin{vmatrix}
	1 & 20 & 5 \\
	0 & 4 & 3 \\
	0 & 0 & 1
	\end{vmatrix}. $$

Again using Table~\ref{tab hihi},
	$$ \Pi_{(2,1,0)}|^{SO(7)}_{SO(3)} = {20} \times \Psi_{(0)} \ + \ 20 \times \Psi_{(1)} \ + \ {5} \times \Psi_{(2)},$$ 
	where the multiplicities are given by the following determinant calculation:
	$$20= 2^2 \cdot \begin{vmatrix}
	\vspace{2mm}	10 & \frac{11}{2} \frac{9}{2} \frac{7}{2} \frac{1}{3!} & \frac{9}{2} \\ \vspace{2mm}
		1 & \frac{7}{2} \frac{5}{2} \frac{3}{2} \frac{1}{3!} & \frac{5}{2} \\ \vspace{2mm}
		0 & \frac{3}{2} \frac{1}{2} \frac{-1}{2} \frac{1}{3!} & \frac{1}{2}
	\end{vmatrix}, \quad 20 =2^2 \cdot \begin{vmatrix}
	\vspace{2mm}	4 & \frac{11}{2} \frac{9}{2} \frac{7}{2} \frac{1}{3!} & \frac{9}{2} \\ \vspace{2mm}
	0 & \frac{7}{2} \frac{5}{2} \frac{3}{2} \frac{1}{3!} & \frac{5}{2} \\ \vspace{2mm}
	0 & \frac{3}{2} \frac{1}{2} \frac{-1}{2} \frac{1}{3!} & \frac{1}{2}
	\end{vmatrix}, \quad 5=2^2 \cdot \begin{vmatrix}
	\vspace{2mm}	1 & \frac{11}{2} \frac{9}{2} \frac{7}{2} \frac{1}{3!} & \frac{9}{2} \\ \vspace{2mm}
	0 & \frac{7}{2} \frac{5}{2} \frac{3}{2} \frac{1}{3!} & \frac{5}{2} \\ \vspace{2mm}
	0 & \frac{3}{2} \frac{1}{2} \frac{-1}{2} \frac{1}{3!} & \frac{1}{2}
	\end{vmatrix}. $$

We omit calculations for other pairs. We summarize and get the following table:

\begin{center}
	\begingroup
	\setlength{\tabcolsep}{8pt} 
	\renewcommand{\arraystretch}{1.75} 
	\captionof{table}{Calculation of $\Pi_{(2,1,0)}|^G_H$ for different groups.}\label{table all}	\begin{tabular}{|c||c|c|c|}
		\hline
		Pair &  \setlength{\baselineskip}{15pt} $m((2,1,0),(0))$ & $m((2,1,0),(1))$ & $m((2,1,0),(2))$ \\
		\hline \hline
		$\Sp^6_2$ & 20 & 16 & 4 \\ \hline
		$\SO^7_3$ & 20 & 20 & 5   \\ \hline
		$\SO^6_2$ & 24 & 16 & 4   \\ \hline	\hline
		$\SO^6_3$ & 8 & 12 & 4   \\ \hline 
		$\SO^7_2$ & 45 & 25 & 5   \\ \hline \hline
		$\GL^3_1$ & 2 & 4 & 2   \\ \hline \end{tabular}
	\endgroup
	\vspace{3mm}
\end{center}

We verify Table~\ref{table all} using Sage \cite{sage}. Within Sage, we use the Branching Rules module \cite{sagebran}.
\end{example}

\section{Preliminaries}\label{sec prelim}
Let $H \subset G$ be classical groups with Lie algebras $\mathfrak{h} \subset \mathfrak{g}$. Let $T_G$ and $T_H$ be maximal algebraic tori in $G$ and $H$ respectively with $T_H \subset T_G$. Let $\mathfrak{t}_{\mathfrak{g}}$ and $\mathfrak{t}_{\mathfrak{h}}$ be the corresponding Lie algebras of torus. Let $\Phi_{\mathfrak{g}}$ and $\Phi_{\mathfrak{h}}$ 
be the roots of  $\mathfrak{g}$ and $\mathfrak{h}$ respectively. Let $\Phi_{\mathfrak{g}}^{+}$ and $\Phi_{\mathfrak{h}}^{+}$ be a system of positive roots for $\mathfrak{g}$ and $\mathfrak{h}$ respectively such that the restriction of a positive root of $\mathfrak{g}$ to $\mathfrak{t}_{\mathfrak{h}}$ is either zero or positive.
After fixing the simple roots \cite[Subsection 2.4.3]{GW} we can describe associated positive roots $\Phi^{+}_{\mathfrak{g}}$ as follows:
\begin{equation}\label{pos root}
	\Phi_{\mathfrak{g}}^{+} = \begin{cases}
		\begin{aligned}
			& \left\{\varepsilon_i -   \varepsilon_j : 1 \leq i<j \leq n\right\}, && \text{ for } \gln, \\
			& \left\{\varepsilon_i \pm   \varepsilon_j : 1 \leq i<j \leq n\right\},  && \text{ for } \SO(2n), \\
			& \left\{\varepsilon_i \pm   \varepsilon_j : 1 \leq i<j \leq n\right\} \cup \{2 \varepsilon_i : 1 \leq i \leq n \}, && \text{ for } \Sp(2n),\\
			&	\left\{\varepsilon_i \pm   \varepsilon_j : 1 \leq i<j \leq n\right\} \cup \{ \varepsilon_i : 1 \leq i \leq n \}, && \text{ for } \SO(2n+1).
		\end{aligned}
	\end{cases} 
\end{equation}
We define  $\rho_{\mathfrak{g}}$, $R_{\mathfrak{g}}$ for $G$ as follows:
\begin{equation}\label{rhog and rg}
	\rho_{\mathfrak{g}}=\frac{1}{2} \sum_{\alpha \in \Phi_{\mathfrak{g}}^{+}} \alpha, \quad \quad R_{\mathfrak{g}}= \prod_{\alpha \in \Phi_{\mathfrak{g}}^{+}}\left(1-\mathrm{e}^{-\alpha}\right). 
\end{equation} 
 We denote \(\rho_{\mathfrak{g}}=\sum_{i=1}^{n} \rho_i \varepsilon_i\). The description of $\rho_i$ for each group $G$  is as follows:
\begin{equation}\label{rho}
	\rho_i= \begin{cases}
		n-i, & \text{ for } \quad \gln ,  \SO(2n), \\
		n-i+1, & \text{ for } \quad \Sp(2n), \\
		n-i+\frac{1}{2}, & \text{ for } \quad \SO(2n+1).
	\end{cases} 
\end{equation}

\begin{defn}{(Weyl Denominator)}
	\begin{equation}\label{delg}
		\Delta_{\mathfrak{g}} = \mathrm{e}^{\rho_{\mathfrak{g}}} \prod_{\alpha \in \Phi_{\mathfrak{g}}^{+}}\left(1-\mathrm{e}^{-\alpha}\right) =\mathrm{e}^{\rho_{\mathfrak{g}}} \cdot R_{\mathfrak{g}}.
	\end{equation}
\end{defn}
We denote the Weyl group of $G$ as $W_{\mathfrak{g}}$. The description of $W_{\mathfrak{g}}$ is as follows:
\begin{align*}
	W_{\mathfrak{g}}=
	\begin{cases}
		\mathfrak{S}_n, & \text{ for } \quad \gln, \\
		(\Z/{2\Z})^{n} \rtimes \mathfrak{S}_n, & \text{ for } \quad  \Sp(2n) ,  \SO(2n+1), \\
		(\Z/{2\Z})^{n-1} \rtimes \mathfrak{S}_n, & \text{ for } \quad  \SO(2n),
	\end{cases}
\end{align*}
where $
(\Z/{2\Z})^{n}=\left\langle\sigma_1, \ldots, \sigma_n\right\rangle$ and  $\sigma_i$ acts as $\sigma_i \varepsilon_j=(-1)^{\delta_{i j}} \varepsilon_j$.

\begin{theorem}\label{weyl char formula}
	(Weyl Character Formula). Let $\lambar$ be a dominant integral weight of $\mathfrak{t}_{\mathfrak{g}}$ and $\pilam$ the corresponding finite-dimensional irreducible $G$-module. Then
	$$
	\Delta_{\mathfrak{g}} \cdot \operatorname{ch}\left(\pilam \right)=\sum_{s \in W_{\mathfrak{g}}} \operatorname{sgn}(s) \mathrm{e}^{s \cdot(\lambar+\rho_{\mathfrak{g}})}.
	$$
\end{theorem}

 Now, we state the Weyl dimension formula for all groups. For proofs and further details, see \cite[Subsection~7.1.2]{GW} and \cite[Chapter 24]{FH}.
\begin{theorem}(Weyl Dimension Formula) \label{weyl dim}
	For a sequence of integers $\lambar=(\lambda_1 \geq \lambda_2 \geq \cdots \geq \lambda_n \geq 0)$ and an irreducible representation $\pilam$ with the highest weight $\lambar$, the dimension of $\pilam$ is given by the following formula.
	\[
		\begin{aligned}
			& \prod_{1 \leq i <j \leq n} \frac{(\lambda_i + j -\lambda_j -i)}{ j-i},   && \hspace*{-2.5mm} \text{ for } \gln, \vspace{1mm} \\
			&	\prod_{1 \leq i <j \leq n} \frac{(\lambda_i + n-i)^2 -(\lambda_j +n-j)^2}{ (n-i)^2 - (n-j)^2},  && \hspace*{-2.5mm} \text{ for } \SO(2n), \vspace{1mm}  \\
			&	\prod_{1 \leq i <j \leq n} \frac{(\lambda_i + n+1-i)^2 -(\lambda_j + n+1-j)^2}{ (n+1-i)^2 - (n+1-j)^2}   \prod_{1 \leq i  \leq n} \frac{\lambda_i + n+1-i}{ n+1-i}, && \hspace*{-2.5mm} \text{ for } \Sp(2n), \\
			&	\prod_{1 \leq i <j \leq n} \frac{(\lambda_i + n+\frac{1}{2}-i)^2 -(\lambda_j + n+\frac{1}{2}-j)^2}{ (n+\frac{1}{2}-i)^2 - (n+\frac{1}{2}-j)^2}   \prod_{1 \leq i  \leq n} \frac{\lambda_i + n+\frac{1}{2}-i}{ n+\frac{1}{2}-i}, && \hspace*{-2.5mm} \text{ for } \SO(2n+1).
		\end{aligned}
	\]
\end{theorem}

\begin{proposition}\label{dim denom}
	For \(\rho_{\mathfrak{g}}=\sum_{i=1}^{n} \rho_i \varepsilon_i\), half the sum of positive roots, we have:
	
	$$
	\begin{aligned}
		\prod_{1 \leq i <j \leq n} (\rho_i - \rho_j)    &=  \prod_{j=1}^{n-1} (n-j)! && \text{ when } \ G=\gln,  \\
		\prod_{1 \leq i <j \leq n} (\rho_i^2 - \rho_j^2)    &= \frac{1}{2^{n-1}} \prod_{j=1}^{n-1} (2n-2j)!  && \text{ when } \ G=\SO(2n), \\
		\prod_{1 \leq i <j \leq n} (\rho_i^2 - \rho_j^2)   \prod_{1 \leq i  \leq n} \rho_i 
		&=  \frac{1}{2^n}  \prod_{j=1}^n (2n-2j+1)!  && \text{ when } \ G= \SO(2n+1),  \\
		\prod_{1 \leq i <j \leq n} (\rho_i^2 - \rho_j^2)   \prod_{1 \leq i  \leq n} \rho_i 
		&=  \prod_{j=1}^n (2n-2j+1)!  && \text{ when } \ G= \Sp(2n). 
	\end{aligned}
	$$
\end{proposition}

We omit the straightforward proof.

\begin{defn}\label{def det}
	Recall that the determinant for an $n \times n$ matrix $A=(a_{ij})$ is:
	\begin{equation*}
		\operatorname{det}(A)= \sum_{\sigma \in \Sn} \operatorname{sgn}(\sigma) \prod_{j=1}^n a_{\sigma(j),j}.
	\end{equation*}
\end{defn}

\begin{proposition}\label{dim num}
	\begin{align*}
		\det \left[x^{n-j}_i\right] &=	\prod_{1 \leq i < j \leq n} (x_i -x_j),  \\
		\det \left[x^{2n-2j}_i \right]  &=	\prod_{1 \leq i < j \leq n} (x_i^2 -x_j^2), \\
		\det \left[x^{2n-2j+1}_i \right] &=	\prod_{1 \leq i < j \leq n} (x_i^2 -x_j^2)  \prod_{1 \leq i  \leq n} x_i.   
	\end{align*}
\end{proposition}
\begin{proof}
	The first identity arises from the determinant of the Vandermonde matrix. The second matrix corresponds to a Vandermonde-type determinant where the values of $x_i$ are substituted with $x_i^2$. A similar procedure can be applied to the third matrix.
\end{proof}

\section{Weyl Dimension Formula as Determinant}
 The following lemma establishes a relationship between determinants of matrices with entries containing factorial terms and matrices with entries as binomial coefficients. This lemma plays a crucial role in formulating and proving the main theorems presented in this paper.

\begin{lemma}\label{det identity}
	We have the following three equalities of determinants of $n \times n$ matrices:
	\begin{align*}
		\det \left[\frac{x_i^{n-j}}{(n-j)!}\right] &= \det \left[\binom{x_i}{n-j}\right],  \\ 
		\det \left[\frac{x_i^{2n-2j+1}}{(2n-2j+1)!}\right] &= \det \left[\binom{x_i+n-j}{2n-2j+1}\right], \\
		 \det \left[\frac{x_i^{2n-2j}}{(2n-2j)!}\right] &= \det \left[\binom{x_i+n-j-\frac{1}{2}}{2n-2j}\right].
	\end{align*}
\end{lemma}

\begin{proof}
	 The first equality is established by applying column operations to the matrix defined by the $(i, j)$-th entry equal to $\frac{x_{i}^{n-j}}{(n-j)!}$. For $1\leq j \leq n-1$, one can write
	$$\binom{x}{n-j}=\frac{x^{n-j}}{(n-j)!} + \sum_{j+1\leq k \leq n-1} a_{j,k} \frac{x^{n-k}}{(n-k)!},$$
	where $a_{j,k}$ is a constant. We perform a column operation on the $j$-th column, where $1\leq j \leq n-1$, given by
	$$C_j^{'}=C_j + \sum_{j+1\leq k \leq n-1} a_{j,k} C_k,$$
	where $C_j$ and $C_j^{'}$ denote the old and new $j$-th column, respectively. Following this operation, the $(i, j)$-th entry transforms into $\binom{x_i}{n-j}$. This establishes the first equality in the lemma. 
	
	For the proof of the other two equalities, we need the following two identities, respectively, where $1\leq j \leq n$:
	\begin{align*}
		\binom{x+n-j}{2n-2j+1} &=\frac{x^{2n-2j+1}}{(2n-2j+1)!} + \sum_{j+1\leq k \leq n} b_{j,k} \frac{x^{2n-2k+1}}{(2n-2k+1)!}, \\
		\binom{x+n-j-\frac{1}{2}}{2n-2j} &=\frac{x^{2n-2j}}{(2n-2j)!} + \sum_{j+1\leq k \leq n} c_{j,k} \frac{x^{2n-2k}}{(2n-2k)!}. 
	\end{align*}
	Here, again, $b_{j,k}$ and $c_{j,k}$ are constants.
\end{proof}
The above lemma is necessary for deriving the following theorem, which expresses the dimension of an irreducible representation with highest weight $\lambar=(\lambda_1 \geq \lambda_2 \geq \cdots \geq \lambda_n \geq 0)$ as a determinant of a certain combinatorial $n \times n$ matrix.

\begin{theorem}(New formulation of Weyl Dimension)\label{new weyl dim}
	$$ \dim(\pilam) =
	\begin{cases}
		\begin{aligned}
			&	\det \left[\binom{\lambda_i +n -i}{n-j}\right], && \text{ for } \quad  \gln,  \\ 
			& 	\det \left[\binom{\lambda_i-i + 2n -j+1}{2n-2j+1}\right],	 &&	 \text{ for } \quad \Sp(2n), \\
			2^n &\det \left[\binom{\lambda_i-i + 2n -j+\frac{1}{2}}{2n-2j+1}\right],	  &&  \text{ for } \quad \SO(2n+1), \\
			2^{n-1} &	\det \left[\binom{\lambda_i-i + 2n -j-\frac{1}{2}}{2n-2j}\right], &&   \text{ for } \quad \SO(2n).
		\end{aligned}
	\end{cases}
	$$
\end{theorem}

\begin{proof}
	We only prove this for the general linear group; the proofs for other groups follow similarly. 
	Using Proposition~\ref{dim num} and Proposition~\ref{dim denom}, Theorem~(\ref{weyl dim}) gives the following expression for the Weyl dimension formula for $\gln$ 
	\[\prod_{1 \leq i <j \leq n} \frac{(\lambda_i + j -\lambda_j -i)}{ j-i}=\frac{\prod_{1 \leq i <j \leq n} \left[(\lambda_i + \rho_i) -(\lambda_j + \rho_j)\right] }{\prod_{1 \leq i <j \leq n} \left( \rho_i - \rho_j \right) } = \frac{\det \left[\left(\lambda_i + \rho_i \right)^{n-j}\right]}{\prod_{j=1}^{n-1} (n-j)!}.\]
	Now we insert $(n-j)!$ inside the determinant in $j$-th column and then apply Lemma~\ref{det identity}.
	\[\frac{\det \left[\left(\lambda_i + \rho_i \right)^{n-j}\right]}{\prod_{j=1}^{n-1} (n-j)!} = \det \left[\frac{\left(\lambda_i + \rho_i \right)^{n-j}}{(n-j)!}\right] = \det \left[\binom{\lambda_i +n -i}{n-j}\right].\]
	This completes the proof of Theorem for the general linear groups. Similar arguments can be given for other classical group.
\end{proof}

\section{Partition Function}\label{sec partition}
Before going into the proof of Theorem~\ref{general gl}-\ref{general ortho}, we derive the necessary partition function for each pair in this section. Note that, we follow the same notation and definitions as in \cite[Subsection~8.2.1]{GW}. We use the following notations as defined earlier:
\[\begin{array}{rl}
	\begin{aligned}
		\glnm \, &:=\\
		\spnm \, &:=\\
		\sopq \, &:=
	\end{aligned} & \begin{aligned}
	\left( \right. \GL(m) &\subset \GL(n) \left. \right),\\
	\left( \right. \Sp(2m) &\subset \Sp(2n) \left. \right),\\
	\left( \right. \SO(q) &\subset \SO(p) \left. \right).
	\end{aligned}
\end{array}\]
Here $p=2n \text{ or } 2n+1 \text{ and } q=2m \text{ or } 2m+1$.
Let $\overline{\alpha}$ be the element of $\mathfrak{t}_{\mathfrak{h}}^{*}$ obtained by restricting $\alpha$ from the Lie algebra $\mathfrak{t}_{\mathfrak{g}}$ to the Lie algebra $\mathfrak{t}_{\mathfrak{h}}$.
Define $\overline{\Phi_{\mathfrak{g}}^{+}}$ as $\left\{\bar{\alpha}: \alpha \in \Phi_{\mathfrak{g}}^{+}\right\}$. Since $\overline{\varepsilon_{m+i}}=0$, for $i\geq 1$ we have $\overline{\Phi_{\mathfrak{g}}^{+}}$ as follows:

\[\overline{\Phi_{\mathfrak{g}}^{+}} = \begin{cases}
	\begin{aligned}
		&	\left\{\varepsilon_i - \varepsilon_j: 1 \leq i<j \leq m \right\} \cup 	\left\{\hspace*{3mm} \varepsilon_i \hspace*{4mm} : 1 \leq i \leq m \right\}, && \text{ for } \, \glnm, \\
		&	\left\{\varepsilon_i \pm \varepsilon_j: 1 \leq i<j \leq m \right\} \cup 	\left\{\varepsilon_i, 2 \varepsilon_i : 1 \leq i \leq m \right\}, \quad && \text{ for } \, \spnm, \\
		&	\left\{\varepsilon_i \pm \varepsilon_j: 1 \leq i<j \leq m \right\} \cup 	\left\{\hspace*{3mm} \varepsilon_i \hspace*{4mm} : 1 \leq i \leq m \right\},  && \text{ for } \, \soodod, \soodev, \\
		&	\left\{\varepsilon_i \pm \varepsilon_j: 1 \leq i<j \leq m \right\} \cup 	\left\{\hspace*{3mm} \varepsilon_i \hspace*{4mm} : 1 \leq i \leq m \right\},  && \text{ for } \, \soevod, \soevev. 
	\end{aligned}
\end{cases}\]

Note that $\Phi_{\mathfrak{h}}^{+}$ is a subset of $\overline{\Phi_{\mathfrak{g}}^{+}}$. For each $\beta \in \overline{\Phi_{\mathfrak{g}}^{+}}$, let $S_\beta$ represent $\left\{\alpha \in \Phi_{\mathfrak{g}}^{+}: \bar{\alpha}=\beta\right\}$. Now, consider the following definitions:
\begin{equation*}\label{sigma set}
	\Sigma_0=\left\{\beta: \beta \in \Phi_{\mathfrak{h}}^{+} \text { and }\left|S_\beta\right|>1\right\}, \quad \Sigma_1=\overline{\Phi_{\mathfrak{g}}^{+}} \backslash \Phi_{\mathfrak{h}}^{+} .
\end{equation*}
Consider the set $\Sigma=\Sigma_0 \cup \Sigma_1$. For each $\beta \in \Sigma$, we define the multiplicity $m_\beta$ as follows:

\begin{equation*}
	m_\beta=  \begin{cases}
		\left|S_\beta\right|, & \text { if } \beta \notin \Phi_{\mathfrak{h}}^{+},\\
		\left|S_\beta\right|-1, & \text { if } \beta \in \Phi_{\mathfrak{h}}^{+} .
	\end{cases}
\end{equation*}
Define the partition function $\wp_{\Sigma}$ on $\mathfrak{t}_{\mathfrak{h}}^*$ by the formal identity
\begin{equation}\label{def partition}
	\frac{1}{R_{\Sigma}}:=\prod_{\beta \in \Sigma}\left(1-\mathrm{e}^{-\beta}\right)^{-m_\beta}=\sum_{\xi} \wp_{\Sigma}(\xi) \mathrm{e}^{-\xi}.
\end{equation}
$\wp_{\Sigma}(\xi)$ is the number of ways of writing
$
\xi=\sum_{\beta \in \Sigma} c_\beta \beta \quad\left(c_\beta \in \mathbb{N}\right)
$
where each $\beta$ that occurs is counted with multiplicity $m_\beta$.

We describe the set $\Sigma$ and the multiplicity $m_\beta$ of an element $\beta$ in $\Sigma$ for all pairs $H \subset G$.
\begin{proposition}\label{sigma and r}
	Consider all branching pairs $H \subset G$. Then 
	\[\Sigma=\{\varepsilon_i : 1\leq i \leq m\}, \, (r:=)m_{\varepsilon_i}=\begin{cases}
		n-m, & \text{ for } \, \glnm, \\
		2n-2m, & \text{ for } \, \spnm, \\
		2n+1-q, &  \text{ for } \, \soodq, q=2m \text{ or } 2m+1, \\
		2n-q, &  \text{ for } \, \soevq, \hspace{4mm} q=2m \text{ or } 2m+1.
	\end{cases}\]
\end{proposition}
Observe that $m_{\varepsilon_i}$ remains constant for $1 \leq i \leq m$ for each group; therefore, we denote this constant by $r$. We do not provide a detailed proof for each pair, as it becomes evident through subsequent calculations.

\begin{lemma}\label{lemma partition}
	Consider the same $\Sigma$ and $r$ as in the preceding proposition. Let $\xibar \in \mathfrak{t}_{\mathfrak{h}}^*$ be defined as $\xibar= \sum_{j=1}^m \xi_j \varepsilon_j$. Then,
	\begin{equation*}\label{partition formula gl}
		\wp_{\Sigma}\left(\xibar \right)= \prod_{j=1}^{m}  \fracnom{\xi_j+r-1}{r-1}.
	\end{equation*}
\end{lemma}
\begin{proof}
	Using the definition of partition function from equation~(\ref{def partition}) to the previously mentioned $\Sigma$ and $r$ in the preceding proposition, our partition function becomes:
	\begin{equation*}\label{part fun}
		\prod_{j=1}^{m}\left(1-\mathrm{e}^{-\varepsilon_j}\right)^{-r}=\sum_{\xibar \in \mathrm{t}_{\mathfrak{h}}^*} \wp_{\Sigma}(\xibar) \mathrm{e}^{-\xibar}.
	\end{equation*}
	So, $\wp_{\Sigma}(\xibar)$ represents the coefficients of $\mathrm{e}^{-\xibar}=\prod_{j=1}^{m} (\mathrm{e}^{-\varepsilon_j})^{\xi_j}$ in $\prod_{j=1}^{m}\left(1-\mathrm{e}^{-\varepsilon_j}\right)^{-r}$. Since the $\varepsilon_j$'s are linearly independent for $1 \leq j \leq m$, it suffices to determine the coefficients of $(\mathrm{e}^{-\varepsilon_j})^{\xi_j}$ in the expression of $\left(1-\mathrm{e}^{-\varepsilon_j}\right)^{-r}$ and then take the product over $j$. Specifically, we need to find the coefficients of $z^{\xi_j}$ in the expression of $\left(1-z\right)^{-r}$. Consider the following power series identity,
	\[\frac{1}{(1-z)^{r}}=\sum_{l \geq 0} \binom{l+r-1}{r-1} z^l.\]
	The coefficients mentioned above are some binomial coefficient. One can prove this identity by induction. Assuming it is true for $r-1$, differentiate both sides of the expression for value $r-1$ to obtain expressions for value $r$. Hence,
	if $\xi_j \geq 0$, the coefficients of $z^{\xi_j}$ are $\binom{\xi_j+r-1}{r-1}$; otherwise, the coefficient are zero. 
	\begin{equation*}\label{partition formula gl}
		\wp_{\Sigma}\left(\xibar \right)= \begin{cases}\prod_{j=1}^{m} \binom{\xi_j+r-1}{r-1}, & \text { if } \hspace*{2mm} \xi_j \geq 0 \hspace*{2mm} \text { for } \hspace*{2mm} 1\leq j \leq m, \\ \hspace*{1cm} 0, & \text { otherwise}.\end{cases}
	\end{equation*}
	Now using definition~\ref{def fracnom}, $\wp_{\Sigma}\left(\xibar \right)$ becomes:
	\begin{equation*}
		\wp_{\Sigma}\left(\xibar \right)= \prod_{j=1}^{m}  \fracnom{\xi_j+r-1}{r-1}.
	\end{equation*}
	This completes the proof of this Lemma.
\end{proof}

\section{Proof of Theorems \ref{general gl} - \ref{general ortho}}\label{sec all}\label{sec gen proof}
In this section, we prove the multiplicity formulae given in Theorem~\ref{general gl}-\ref{general ortho}. Our proof strategies follow the approach used in the restriction from $\Sp(2n)$ to $\Sp(2n-2)$ as in \cite[Subsection~8.3.4]{GW}.

 Recall the notation $R_{\mathfrak{h}}$ from Section~\ref{sec prelim}. The main idea of the proof is to use Weyl character formula (Theorem~\ref{weyl char formula}) for both $H$ and $G$ to get two expression for $R_{\mathfrak{h}} \cdot \pilam|_H$. By equating the coefficient of $\mathrm{e}^{\mubar}$ in these two expression, we find the multiplicity $m(\lambar, \mubar)$. We divide the proof into several parts.
\vspace{3mm}

\noindent \textbf{First Part}  We find the coefficient of $\operatorname{e}^{\mubar}$ in $R_{\mathfrak{h}} \cdot \pilam|_H$ using the Weyl character formula for $H$ in this part. 
Using the Weyl character formula for $H$ to $\Psi_{\mubar}$, equation~(\ref{restrict}) becomes:
\begin{equation}\label{weyl char psi gl}
	R_{\mathfrak{h}} \cdot \pilam|_H =\sum_{\mubar} \sum_{w \in W_{\mathfrak{h}}} m(\lambar, \mubar)  \operatorname{sgn}(w)  \mathrm{e}^{ w \cdot\left(\mubar + \rho_{\mathfrak{h}}\right)- \rho_{\mathfrak{h}}}.
\end{equation}
Consider two dominant weights $\mubar$ and $\mubard$ that appear in the sum. Since each weight is conjugate under the Weyl group $W_{\mathfrak{h}}$ to exactly one dominant weight, the equation \[ w \cdot\left(\mubar + \rho_{\mathfrak{h}}\right)- \rho_{\mathfrak{h}} = \mubard \] implies $\mubar' =\mubar$. Additionally, as $\mubar+\rho_{\mathfrak{h}}$ is regular, the equation $w \cdot\left(\mubar + \rho_{\mathfrak{h}}\right)- \rho_{\mathfrak{h}} = \mubar$ implies $w=1$. Consequently, the coefficient of $\operatorname{e}^{\mubar}$ in equation~(\ref{weyl char psi gl}) is $m(\lambar, \mubar)$. 
\vspace{3mm}

\noindent \textbf{Second Part} We find an expression of Weyl denominator at the end of this part, which is needed in the next part.
Consider the set $\Phi_{\mathfrak{g}}^{+}$ (see equation~(\ref{pos root})) of positive roots of $G$. Recall the notation $\overline{\alpha}$. Define $P$ and $Q$ to be subsets of $\Phi_{\mathfrak{g}}^{+}$ defined as follows:
\[ P = \{ \alpha \in \Phi_{\mathfrak{g}}^{+} : \overline{\alpha} = 0 \}, \quad  \quad Q = \{ \alpha \in \Phi_{\mathfrak{g}}^{+} : \overline{\alpha} \neq 0 \}. \]
So  $\Phi_{\mathfrak{g}}^{+}$ is the disjoint union of $P$ and $Q$. 
We define $\rho_s$, $R_s$ for a set $S \subset \Phi_{\mathfrak{g}}^{+}$  as follows analogously in equation~(\ref{rhog and rg}). 
\[ \rho_s=\frac{1}{2} \sum_{\alpha \in S} \alpha,  \quad \quad R_s= \prod_{\alpha \in S}\left(1-\mathrm{e}^{-\alpha}\right). \]
 Hence 
 \begin{equation}\label{rg prod}
 	R_{\mathfrak{g}}=R_p \cdot R_{q}, \quad \text{ and } \quad \mathrm{e}^{\rho_{\mathfrak{g}}} = \mathrm{e}^{\rho_p} \cdot \mathrm{e}^{\rho_q}.
 \end{equation}
 
Here is the exact description of $P$ for each pairs.
$$
P =  \begin{cases}
	\begin{aligned}
		& \left\{\varepsilon_i -   \varepsilon_j : m+1 \leq i<j \leq n\right\}, && \hspace*{-3.5mm} \text{ for } \  \glnm, \\
		& \left\{\varepsilon_i \pm   \varepsilon_j : m+1 \leq i<j \leq n\right\},  && \hspace*{-3.5mm} \text{ for }\, \, \soevod, \soevev, \\
		& \left\{\varepsilon_i \pm   \varepsilon_j : m+1 \leq i<j \leq n\right\} \cup \{2 \varepsilon_i : m+1 \leq i \leq n \}, && \hspace*{-3.5mm} \text{ for }  \ \spnm, \\
		&	\left\{\varepsilon_i \pm   \varepsilon_j : m+1 \leq i<j \leq n\right\} \cup \{ \varepsilon_i : m+1 \leq i \leq n \}, && \hspace*{-3.5mm} \text{ for }\, \soodod, \soodev.
	\end{aligned}
\end{cases}
$$
 Observe that the set $P$ can be considered as the set of positive roots of $K$, where $K$ is as follows:
$$
K \cong \begin{cases}
	\begin{aligned}
		& \GL(n-m), && \text{ for } \quad  \glnm, \\
		&  \Sp(2n-2m), && \text{ for } \quad  \spnm, \\
		& \SO(2n-2m+1), && \text{ for } \quad  \soodod, \soodev,  \\
		&  \SO(2n-2m), && \text{ for } \quad  \soevod, \soevev.
	\end{aligned}
\end{cases}
$$
Note that, we can write $R_p = R_{\mathfrak{k}}$ and $\mathrm{e}^{\rho_p} =\mathrm{e}^{\rho_{\mathfrak{k}}}$ as $P=\Phi_{\mathfrak{k}}^{+}$. Therefore, by equation~(\ref{rg prod}),
$R_{\mathfrak{g}}=R_{\mathfrak{k}} \cdot R_{q}$ and $\mathrm{e}^{\rho_{\mathfrak{g}}} = \mathrm{e}^{\rho_{\mathfrak{k}}} \cdot \mathrm{e}^{\rho_q}$. Hence, using equation~(\ref{delg}), the denominator $\Delta_{\mathfrak{g}}$  becomes   
\begin{equation}\label{delta g}
	\Delta_{\mathfrak{g}}= \mathrm{e}^{\rho_q} \cdot R_{q} \cdot \mathrm{e}^{\rho_{\mathfrak{k}}} \cdot R_{\mathfrak{k}}.
\end{equation}

\noindent \textbf{Third Part} In this part,  we use the Weyl character formula for $G$ and break Weyl numerator into the cosets over $W_{\mathfrak{k}} \setminus W_{\mathfrak{g}}$. Further, we use the Weyl dimension formula (Theorem~\ref{new weyl dim}) and get an expression of the multiplicity of $m(\lambar, \mubar)$ as an alternating sum over $W_{\mathfrak{k}} \setminus W_{\mathfrak{g}}$ after finding the coefficient of $\operatorname{e}^{\mubar}$ in certain equation. 

Using the Weyl character formula for $G$ to $\pilam$ and  equation~(\ref{delta g}) we obtain the following:
\begin{equation}\label{end part2 gl}
	\pilam = \frac{1}{  \mathrm{e}^{\rho_q} \cdot R_{q} \cdot \mathrm{e}^{\rho_{\mathfrak{k}}} \cdot R_{\mathfrak{k}} }   \sum_{s \in W_{\mathfrak{g}}} \operatorname{sgn}(s)  \mathrm{e}^{ s \cdot\left(\lambar + \rho_{\mathfrak{g}}\right)}.
\end{equation}
We break down this Weyl numerator into the cosets over $W_{\mathfrak{k}} \setminus W_{\mathfrak{g}}$ as follows:  
\begin{align} \label{sum break all} 
	\sum_{s \in \wg} \operatorname{sgn}(s)  \mathrm{e}^{ s \cdot\left(\lambar + \rho_{\mathfrak{g}}\right)} =  \sum_{s \in \wk 
		\setminus \wg} \operatorname{sgn}(s) \left\{ \sum_{w \in \wk} \operatorname{sgn}(w) \mathrm{e}^{ (ws) \cdot\left(\lambar + \rho_{\mathfrak{g}}\right)}
	\right\}.
\end{align}
Let $\gamma=s \cdot\left(\lambda+\rho_{\mathfrak{g}}\right)$ for $s \in \wg$. We denote \(\gamma=\sum_{i=1}^{n} \gamma_i \varepsilon_i\). Note that, $s \cdot \varepsilon_i = \varepsilon_{s^{-1}(i)}$. So, 
\begin{eqnarray}\label{gamma i}
	\gamma_i=\lambda_{s(i)}+ \rho_{s(i)}.
\end{eqnarray}
As $\overline{\varepsilon_{m+i}}=0$ for $i \geq 1$, we can express $\overline{\gamma}$ as $\overline{\gamma}=\sum_{i=1}^{m} \gamma_i \varepsilon_i$.
Let $\widetilde{\gamma}=\sum_{i=m+1}^{n} \gamma_i \varepsilon_i$. This allows us to express $\gamma = \overline{\gamma} + \widetilde{\gamma}$.
Further, $\omega(\gamma)=\overline{\gamma} + \omega (\widetilde{\gamma})$ for  $w \in \wk$ as $\omega(i)=i$, for $1\leq i \leq m$. Thus, 
\begin{align}\label{pull out}
	\sum_{w \in \wk} \operatorname{sgn}(w) \mathrm{e}^{ (ws) \cdot\left(\lambar + \rho_{\mathfrak{g}}\right)}
	&=   \mathrm{e}^{ \overline{\gamma}} \left\{ \sum_{w \in \wk} \operatorname{sgn}(w) \mathrm{e}^{ w \cdot \widetilde{\gamma} }
	\right\}.
\end{align}

 Observe that,  $\rho_{\mathfrak{k}}=\sum_{i=m+1}^{n} \rho_i \varepsilon_i$ by equation~(\ref{rho}). Let $\tilde{\lambda}=\sum_{i=m+1}^{n} \lambda_i \varepsilon_i $. Note that, $s(i)=i$ for $s \in \wk \setminus \wg$, $m+1 \leq i \leq n$. Hence, by using (\ref{gamma i}), we have $\gamma_i=\lambda_i + \rho_i$. Therefore,
 \begin{equation}\label{gamma tilde}
 	\widetilde{\gamma}= \sum_{i=m+1}^{n} \gamma_i \varepsilon_i=\sum_{i=m+1}^{n} (\lambda_i + \rho_i) \varepsilon_i=\sum_{i=m+1}^{n} \lambda_i \varepsilon_i  + \sum_{i=m+1}^{n} \rho_i \varepsilon_i=\tilde{\lambda} + \rho_{\mathfrak{k}}.
 \end{equation} 
Equation (\ref{end part2 gl}) becomes the following after combining equations (\ref{gamma tilde}), (\ref{pull out}), (\ref{sum break all})
\begin{equation}\label{weyl char k}
	\pilam = \frac{1}{\mathrm{e}^{\rho_{q}} \cdot R_{q}  }     \sum_{s \in \wk 
		\setminus \wg} \operatorname{sgn}(s) \mathrm{e}^{ \overline{\gamma}} \left\{ \frac{\sum_{w \in \wk} \operatorname{sgn}(w) \mathrm{e}^{ w \cdot\left( \tilde{\lambda} + \rho_{k}\right)}}{\mathrm{e}^{\rho_{\mathfrak{k}}} \cdot R_{\mathfrak{k}}} 
	\right\}. 
\end{equation}
The expression within the parentheses in (\ref{weyl char k}) corresponds to the Weyl character formula for $K$. Now when we restrict $ \pilam$ from $\mathfrak{t}_{\mathfrak{g}}$ to $\mathfrak{t}_{\mathfrak{h}}$, the expression inside the parentheses provides $\dim(\chi_{\tilde{\lambda}})$, the Weyl dimension of the representation $\chi_{\tilde{\lambda}}$ of $K$ with the highest weight $\tilde{\lambda}$. Observe that $\overline{Q}=\overline{\Phi_{\mathfrak{g}}^{+}}$, implies $\mathrm{e}^{\overline{\rho_{q}}}=\mathrm{e}^{ \overline{\rho_{\mathfrak{g}}}}$. Further 
$\overline{\Phi_{\mathfrak{g}}^{+}}=\Phi_{\mathfrak{h}}^{+} \cup \Sigma$ implies  $\overline{R_q}=R_{\mathfrak{h}} \cdot R_{\Sigma}$.
Hence taking the restriction of the equation~(\ref{weyl char k}) we get,
\begin{eqnarray*}
	\pilam|_H =\frac{1}{\mathrm{e}^{ \overline{\rho_{\mathfrak{g}}}} \cdot R_{\mathfrak{h}} \cdot R_{\Sigma}}     \sum_{s \in \wk 
			\setminus \wg} \operatorname{sgn}(s) \mathrm{e}^{ \overline{\gamma}} \dim(\chi_{\tilde{\lambda}}).
\end{eqnarray*}
 Using equation~(\ref{def partition}), the last equation can be written as:
\begin{align}\label{part3 end all}
	R_{\mathfrak{h}} \cdot \pilam|_H 
	= \sum_{s \in \wk 
		\setminus \wg} \sum_{\xi \in \mathfrak{t}_{\mathfrak{h}}^*} \operatorname{sgn}(s) \dim(\chi_{\tilde{\lambda}}) \wp_{\Sigma}(\xi)  \mathrm{e}^{\overline{\gamma}-\xi-\overline{\rho_{\mathfrak{g}}}}. 
\end{align}
Equate the coefficient of $\mathrm{e}^{\mubar}$ in (\ref{part3 end all}) with that in (\ref{weyl char psi gl}) to obtain:
\begin{equation}\label{end part3}
	m(\lambar, \mubar)	= \sum_{s \in \wk 
		\setminus \wg}  \operatorname{sgn}(s) \dim(\chi_{\tilde{\lambda}}) \wp_{\Sigma}(\overline{\gamma}-\mubar-\overline{\rho_{\mathfrak{g}}}).
\end{equation}

\vspace{1mm}

\noindent \textbf{Fourth Part} In this part, we write the summation in equation~(\ref{end part3})  over $\mathfrak{S}_{n-m} \setminus \Sn$. Then we use Weyl dimension formula for $K$ and partition value to get the final formula.

 We can identify $\wk 
\setminus \wg$ as follows:
\begin{equation*}
	\begin{aligned}
		& \ \  \mathfrak{S}_{n-m} \setminus \mathfrak{S}_n,  &&  \text{ for }  \glnm, \\
	&	\left((\Z/{2\Z})^{n-m} \rtimes \mathfrak{S}_{n-m}\right) \setminus \left(  (\Z/{2\Z})^{n} \rtimes \mathfrak{S}_n \right), &&   \text{ for }  \spnm, \soodod, \soodev, \\
	&	\left((\Z/{2\Z})^{n-m-1} \rtimes \mathfrak{S}_{n-m}\right) \setminus \left(  (\Z/{2\Z})^{n-1} \rtimes \mathfrak{S}_n \right), &&   \text{ for }   \soevod, \soevev.
	\end{aligned}
\end{equation*}
A set of representatives of $ \wk 
\setminus \wg$ can be taken to be a pair of representatives of $\mathfrak{S}_{n-m} \setminus \mathfrak{S}_n $ and $(\Z/{2\Z})^{n-m} \setminus (\Z/{2\Z})^{n}$.
We already have the sum in (\ref{end part3}) over $\mathfrak{S}_{n-m} \setminus \mathfrak{S}_n $ for $\glnm$. We need to do for other pairs.
Given $s \in \wk 
\setminus \wg$, we can write $s=\sigma \cdot \Bar{v}$, where $\sigma \in (\Z/{2\Z})^{m}$ and $\Bar{v} \in \mathfrak{S}_{n-m} \setminus \mathfrak{S}_n$.
Choose any representative $v \in \Sn$ for $\Bar{v}$. Thus, 
$$\overline{v \cdot\left(\lambda+\rho_{\mathfrak{g}}\right)}=\sum_{j=1}^m (\lambda_{v(j)}+\rho_{v(j)}) \varepsilon_j.$$
Note that each coordinate of $\overline{v \cdot\left(\lambda+\rho_{\mathfrak{g}}\right)}$ is positive. Further, if   
$\overline{\rho_{\mathfrak{g}}}=(\rho_1, \rho_2, \ldots, \rho_m)$, 
 each coordinate of $\overline{\rho_{\mathfrak{g}}}$ is also positive. As every coordinate of $\mubar$ is non-negative, and
consequently, if $\sigma \in (\Z/{2\Z})^{m}$ and $\sigma \neq 1$, it implies that at least one component of $\overline{(\sigma v) \cdot\left(\lambda+\rho_{\mathfrak{g}}\right)}-\mu-\overline{\rho_{\mathfrak{g}}}$ is negative. Hence by Lemma \ref{lemma partition},
$$\wp_{\Sigma}\left(\overline{(\sigma v) \cdot\left(\lambda+\rho_{\mathfrak{g}}\right)}-\mu-\overline{\rho_{\mathfrak{g}}}\right)=0 \quad \text{ for } \sigma \neq 1, \Bar{v} \in (\mathfrak{S}_{n-m} \setminus \Sn).$$ Thus, we may take $\sigma=1$, and we can express the summation over $\mathfrak{S}_{n-m} \setminus \Sn$. Therefore, equation~(\ref{end part3}) becomes:
\begin{equation}\label{coset over sn}
	m(\lambar, \mubar)	=\sum_{s \in \mathfrak{S}_{n-m} \setminus \Sn} \operatorname{sgn}(s) \dim(\chi_{\tilde{\lambda}}) \wp_{\Sigma}(\overline{\gamma}-\mubar-\overline{\rho_{\mathfrak{g}}}).
\end{equation}
Using the Weyl dimension formula (Theorem~\ref{new weyl dim}), equation~(\ref{coset over sn}) becomes:
\begin{equation}\label{sum  before final}
	m(\lambar, \mubar)	= \sum_{s \in \mathfrak{S}_{n-m} \setminus \Sn}  \operatorname{sgn}(s) \left(\sum_{\sigma \in \mathfrak{S}_{n-m}} \operatorname{sgn}(\sigma) \prod_{j=m+1}^{n} g_{\sigma(j),j} \right)
	 \wp_{\Sigma}(\overline{\gamma}-\mubar-\overline{\rho_{\mathfrak{g}}}). 
\end{equation}
As $s(j)=j$ for $m+1\leq j \leq n$, we have $g_{\sigma s(j),j}=g_{\sigma(j),j}$ for $m+1\leq j \leq n$. Hence equation~(\ref{sum  before final}) becomes:
\begin{align*}
	m(\lambar, \mubar)	
	&= \sum_{s \in \mathfrak{S}_{n-m} \setminus \Sn}  \operatorname{sgn}(s) \sum_{\sigma \in \mathfrak{S}_{n-m}} \operatorname{sgn}(\sigma) \prod_{j=m+1}^{n} g_{\sigma s(j),j} \hspace{2mm}
	\wp_{\Sigma}(\overline{\gamma}-\mubar-\overline{\rho_{\mathfrak{g}}}),\\
	&= \sum_{s \in \Sn}  \operatorname{sgn}(s)  \prod_{j=m+1}^{n} g_{ s(j),j} \hspace{2mm}
	\wp_{\Sigma}(\overline{\gamma}-\mubar-\overline{\rho_{\mathfrak{g}}}), \\
	&= \sum_{s \in \Sn}  \operatorname{sgn}(s) \left(\prod_{j=1}^{m} f_{ s(j),j}\right) \hspace{1mm} \prod_{j=m+1}^{n} g_{ s(j),j}.    
\end{align*}
We have used Lemma~\ref{lemma partition} in the last equality. Since 
$$\overline{\gamma}-\mubar-\overline{\rho_{\mathfrak{g}}}=\sum_{j=1}^{m}\left(\lambda_{s(j)}-\mu_j +j -s(j)\right)\varepsilon_j = \sum_{j=1}^{m}u_{s(j),j} \varepsilon_j ,$$
we have $f_{ s(j),j}=\fracnom{u_{s(j),j} +r-1}{r-1}$. Therefore,
 we get the desired matrix as in the statement of the Theorems~\ref{general gl}-\ref{general ortho}. Hence,
this gives us the corresponding multiplicity $m(\lambar, \mubar)$ formulae of Theorems~\ref{general gl}-\ref{general ortho}.

\noindent \textbf{Fifth Part}
Now we find the corresponding interlacing condition for three pair and we prove this by induction. We only give details for  the general linear groups; for other pairs follow similarly.

We prove the multiplicity
$m(\lambar, \mubar)$ is nonzero if and only if 
\begin{equation}\label{interlace gl}
	\lambda_i \geq \mu_i \geq \lambda_{i+n-m} \quad \text{ for } \quad 1 \leq i \leq m,
\end{equation}
where 
	\begin{equation*}
		M_{ij} = \begin{cases}
			\displaystyle\fracnom{u_{ij} + n-m-1}{n-m-1}, & \text{ if } \quad \hspace*{8.5mm} 1 \leq j \leq m, \vspace{2mm}\\
			\displaystyle\binom{u_{ij} + n-j}{n-j}, & \text{ if } \quad  m+1 \leq j \leq n.
		\end{cases} 
	\end{equation*}
This part is a generalization of the induction idea as in \cite[Lemma~8.3.3]{GW}.

We shall prove equation~(\ref{interlace gl}) by induction on $n$ for a fixed $m$, where $n>m$.
We assume the statement is true for $n-1$ and prove for $n$. We have different cases depending upon the position of $\mu_1$ relative to $(\lambda_1, \lambda_2, \ldots, \lambda_{n-m+1})$. 
We need to show the multiplicity $m(\lambar, \mubar)$ is zero when $\mu_1 >\lambda_1$ and $\mu_1 < \lambda_{n-m+1}$, and that $m(\lambar, \mubar)$ is non-zero when $\lambda_i \geq \mu_1 \geq \lambda_{i+1}$ for $1\leq i \leq n-m$. 

 When $\mu_1 > \lambda_1$, the first column is entirely zero, leading to a zero determinant and consequently,  $m(\lambar, \mubar)$ is zero.  When $\mu_1 < \lambda_{n-m+1}$, the first $n-m+1$ rows are linearly dependent, hence multiplicity is zero in this case also. 

The proof is similar in all cases $\lambda_i \geq \mu_1 \geq \lambda_{i+1}$ for $1\leq i \leq n-m$. We give arguments only for $\lambda_1 \geq \mu_1 \geq \lambda_{2}$ case. In this case, the multiplicity $m(\lambar, \mubar)$ becomes $\fracnom{u_{11} + n-m-1}{n-m-1}$ times the determinant of $(n-1)\times(n-1)$ matrix, getting after deleting the first row and the first column. By induction hypothesis, the determinant of this $(n-1)\times(n-1)$ matrix is non-zero. Further,  $\fracnom{u_{11} + n-m-1}{n-m-1}$ is also non-zero as $\lambda_1 \geq \mu_1$. Hence $m(\lambar, \mubar)$ is non-zero in the case  $\lambda_1 \geq \mu_1 \geq \lambda_{2}$. This concludes the proof of interlacing condition for general linear groups.

This completes the proof of Theorems \ref{general gl}-\ref{general ortho}.

\section{Comparison of Multiplicities}

We follow the same notation as in Section~\ref{sec partition}. Define $\ell(\lambar)$, the length of $\lambar$, to be the largest integer $s$ such that $\lambda_s\neq 0$. In this section, we consider the branching for the following pairs (For consistency of the four pairs, note that $2m-n=n-2(n-m)$):

$$
\begin{array}{cc}
	\begin{aligned}
		{\GL}^{n}_{2m-n}&=\left( \right. \GL(2m-n) \subset \GL(n) \left. \right),\\
		{\Sp}^{2n}_{2m} &=\left( \right. \Sp(2m) \subset \Sp(2n) \left. \right),
	\end{aligned} & \begin{aligned}
	{\SO}^{2n+1}_{2m+1} &= \left( \right. \SO(2m+1) \subset \SO(2n+1) \left. \right),\\
	{\SO}^{2n}_{2m}&=\left( \right. \SO(2m) \subset \SO(2n) \left. \right).
	\end{aligned}
\end{array}
$$
 The corollary derived from Theorems \ref{general gl} - \ref{general ortho} is as follows.
\begin{corollary}\label{cor compare}
	\begin{enumerate}
		\item Let $\frac{n}{2} \leq m < n$. For fixed pair ($\lambar$, $\mubar$) with $\ell(\mubar) \leq 2m-n$, the branching multiplicity $m(\lambar, \mubar)$ is independent of the pairs considered.
		\item Let $\frac{n}{2} \leq m < n$. For fixed pair ($\lambar$, $\mubar$) with $\ell(\lambar) \leq 2m-n$, the branching multiplicity $m(\lambar, \mubar)$ is independent of the pairs considered.
		\item Let $0 \leq m < n$. For fixed pair ($\lambar$, $\mubar$) with $\ell(\lambar) \leq m$, the branching multiplicity $m(\lambar, \mubar)$ is independent of the pairs $\Sp^{2n}_{2m}$, $\SO^{2n+1}_{2m+1}$, $\SO^{2n}_{2m}$.
		\end{enumerate}
\end{corollary}

\begin{center}
	\begingroup
	\setlength{\tabcolsep}{8pt} 
	\renewcommand{\arraystretch}{1.75} 
		\captionof{table}{Calculation of $m(\lambar, \mubar)$ for $\lambar=(3,1,0,0)$.}
	\begin{tabular}{|c||c|c|c|c|c|c|c|}
		\hline
		\backslashbox{$(H,G)$}{$\mubar$} &  \setlength{\baselineskip}{15pt} $(0,0)$ & $(1,0)$ &  $(1,1)$ & $(2,0)$ & $(2,1)$ & $(3,0)$ & $(3,1)$ \\
		\hline \hline
		$\Sp^8_4$ & 45 & 40 & 10 & 16& 4&4 &1 \\ \hline
		$\SO^9_5$ & 45 & 40 & 10 & 16& 4&4 &1 \\ \hline
		$\SO^8_4$ & 45 & 40 & 10 & 16& 4&4 &1 \\ \hline  	\end{tabular}
		\begin{tikzpicture}[overlay]
			\draw[black] (-9.78,-.45) ellipse (.55cm and 1.24cm);
			\draw[black] (-8.27,-.45) ellipse (.55cm and 1.24cm);
			\draw[black] (-6.77,-.45) ellipse (.55cm and 1.24cm);
			\draw[black] (-5.29,-.45) ellipse (.55cm and 1.24cm);
			\draw[black] (-3.82,-.45) ellipse (.55cm and 1.24cm);
			\draw[black] (-2.32,-.45) ellipse (.55cm and 1.24cm);
			\draw[black] (-.82,-.45) ellipse (.55cm and 1.24cm);
		\end{tikzpicture}
	\label{tab cor branch}
	\endgroup
	\vspace{3mm}
\end{center}

 Part (1) of the Corollary~\ref{cor compare} determines multiplicities that are always independent across all pairs for a given value of $\lambar$.  
In part (2), the corollary gives a criteria for $\lambar$ that leads to the independence of all multiplicities among pairs. One can find similar results in \cite[Theorem~2]{MW}, proving branching multiplicity is independent of the pairs ${\GL}^{n+1}_{n-1} $ and  ${\Sp}^{2n}_{2n-2}$. As a corollary, Miller
proves that multiplicities are independent of the pairs ${\GL}^{n}_{n-2} $ and  ${\Sp}^{2n}_{2n-2}$, when $\ell(\mubar) \leq n-2$. Table~\ref{tab cor branch} verifies part (3) of Corollary~\ref{cor compare}.

We can describe the multiplicities as a product formula, when $m=n-1$ for the mentioned pairs at the start of this section. 

\begin{corollary}\label{cor prod}
	  The multiplicity 
	$m(\lambar, \mubar)$ is nonzero if and only if 
	\begin{equation*}\label{interlacing formula}
		\lambda_j \geq \mu_j \geq \lambda_{j+2} \quad \text { for } j=1, \ldots, n-1
	\end{equation*}
When these inequalities are satisfied, let
	$$
	x_1 \geq y_1 \geq x_2 \geq y_2 \geq \cdots \geq x_{n-1} \geq y_{n-1} \geq x_{n}
	$$
	be the non-increasing rearrangement of $\left\{\lambda_1, \ldots, \lambda_n, \mu_1, \ldots, \mu_{n-1}\right\}$.
 Then, 
	$$ m(\lambar, \mubar)=  \begin{cases}
		\begin{array}{lc}
			\left[\prod_{j=1}^{n-1}(x_j-y_j+1)\right], &  \text{ for } \quad {\GL}^{n}_{n-2}, \vspace{.2cm} \\ 
			\left[\prod_{j=1}^{n-1}(x_j-y_j+1)\right] (x_n+1),&  \text{ for } \quad {\Sp}^{2n}_{2n-2}, \vspace{.2cm} \\ 
			\left[\prod_{j=1}^{n-1}(x_j-y_j+1)\right](2x_n+1), & \text{ for } \quad {\SO}^{2n+1}_{2n-1}, \vspace{.2cm} \\ 
			\left[\prod_{j=1}^{n-1}(x_j-y_j+1)\right], & \text{ for } \quad {\SO}^{2n}_{2n-2}.
		\end{array}
	\end{cases}
	$$
\end{corollary}

In Corollary~\ref{cor prod}, $m(\lambar, \mubar)$ is known in this form for the pair ${\Sp}^{2n}_{2n-2}$ in \cite[Theorem 8.1.5]{GW}. Such a corollary for all pairs considered above seems new. It is a direct consequence of Theorem~\ref{general gl} (in case of ${\GL}^{n}_{n-2}$) and Theorem~\ref{general ortho} (in case of ${\SO}^{2n+1}_{2n-1}$ and of ${\SO}^{2n}_{2n-2}$), so we omit the proof.

\begin{remark}
	The multiplicity for the restriction problem from $\Sp(2n)$ to $\Sp(2n-2)$ need not be 1, and is a bit complicated . The content of the above corollary in this case is that if $\mubar = (\mu_1 \geq \mu_2 \geq \cdots \geq \mu_{n-1} \geq 0)$, and $\lambar = (\lambda_1 \geq \lambda_2 \geq \cdots \geq \lambda_n \geq 0)$ are highest weight for  $\Sp(2n-2)$ and $\Sp(2n)$, then if $\mu_{n-1}=0$, forcing $x_n=0$ in the notation of Corollary~\ref{cor prod} and therefore 
	$$m_{\Sp(2n-2)}(\lambar, \mubar)=m_{\GL(n-2)}(\lambar, \mubar),$$
	 where now $\lambar$ is considered as a highest weight of $\gln$ and $\mubar$ as a highest weight of $\GL(n-2)$ as $\mu_{n-1}=0$. More generally, we have the following equalities: 
	 $$m_{\SO(2n-2)}(\lambar, \mubar)=m_{\Sp(2n-2)}(\lambar, \mubar)=m_{\SO(2n-1)}(\lambar, \mubar)=m_{\GL(n-2)}(\lambar, \mubar),$$
	  where  $\mubar = (\mu_1 \geq \mu_2 \geq \cdots \geq \mu_{n-1} \geq 0)$ with $\mu_{n-1}=0$.
\end{remark}
  
  In Corollary~\ref{cor prod}, observe that $x_n$ can be either $\lambda_n$ or $\mu_{n-1}$ depending on the specific inequalities. For the pair ${\GL}^{n}_{n-2}$, note that $x_n$ must be zero since $\mu_{n-1}=0$. 
   Further, note that the multiplicity formula $m(\lambar, \mubar)$ for the pair ${\SO}^{2n}_{2n-2}$ does not depend on $x_n$.

\vspace{3mm}
\noindent{\bf Acknowledgement:}
This project is part of the author's thesis completed at IIT Bombay, and the author expresses gratitude for the PhD fellowship provided by the institute.

The author sincerely thanks Prof. Dipendra Prasad, the Ph.D. supervisor, for continuous guidance, support, and insightful discussions. Prof. Prasad's kindness and invaluable feedback on the research work are also appreciated. Special thanks for dedicating substantial time to review the paper and rectify errors.


\end{document}